\documentclass{amsart}
\usepackage{amsmath,amscd,amsthm,amsfonts,amssymb}

\theoremstyle{plain}
\newtheorem{theorem}                 {Theorem}      [section]

\theoremstyle{definition}
\newtheorem{example}      [theorem]  {Example}

\newtheorem{definition}   [theorem]  {Definition}

\numberwithin{equation}{section}

\def \theo-intro#1#2 {\vskip .25cm\noindent{\bf Theorem #1\ }{\it #2}}

\def \rn{\mathbb R}
\def \cn{\mathbb C}

\def \H{\mathcal H}

\def \V{\mathcal V}

\def \sol{\mathfrak{sol}}

\def \lb#1#2{[#1,#2]}

\def \g{\mathfrak{g}}
\def \h{\mathfrak{h}}

\def \n{\mathfrak{n}}

\def \SLR#1{\text{\bf SL}_{#1}(\rn)}
\def \slr#1{\mathfrak{sl}_{#1}(\rn)}

\def \SU#1{\text{\bf SU}(#1)}
\def \su#1{\mathfrak{su}(#1)}

\def \nab#1#2{\hbox{$\nabla$\kern -.3em\lower 1.0 ex
\hbox{$#1$}\kern -.1 em {$#2$}}}

\begin{document}
\baselineskip 22pt \larger
\allowdisplaybreaks

\begin{footnotesize}
\begin{flushright}
CP3-ORIGINS-2010-8
\end{flushright}
\end{footnotesize}
\bigskip

\title{On the Existence of Harmonic morphisms\\
from three-dimensional Lie groups}

\author{Sigmundur Gudmundsson}
\author{Martin Svensson}



\address
{Department of Mathematics, Faculty of Science, Lund University,
Box 118, S-221 00 Lund, Sweden}
\email{Sigmundur.Gudmundsson@math.lu.se}

\address
{Department of Mathematics \& Computer Science, and $\text{CP}^3$-Origins Centre of Excellence for Particle Physics Phenomenology, University of
Southern Denmark, Campusvej 55, DK-5230 Odense M, Denmark}
\email{svensson@imada.sdu.dk}

\dedicatory{To Professor John $\cn$ Wood on his sixtieth birthday.}

\begin{abstract}
In this paper we classify those three-dimensional Riemannian Lie
groups which admit harmonic morphisms to surfaces. 
\end{abstract}

\maketitle

\section{Introduction}

A harmonic morphism between two Riemannian manifolds is a map
with the property that its composition with any local harmonic
function on the target manifold is a local harmonic function
on the domain. These maps can be seen as an extension of conformal
mappings between Riemann surfaces. Harmonic morphisms were
introduced by C.~G.~J.~Jacobi \cite{Jacobi}, but the first
characterization of these in the context of Riemannian manifolds
was made by B.~Fuglede and T.~Ishihara \cite{Fug-1,T-Ish}.
Not surprisingly, harmonic morphisms must solve a non-linear,
overdetermined system of partial differential equations.
Therefore, there is no general existence theory; in fact,
there are examples of Riemannian manifolds which do not allow
any global, non-constant harmonic morphism between them \cite{Ville}.

With this in mind, it is natural to consider low-dimensional
situations when attempting to classify harmonic morphisms.
A submersive harmonic morphism gives rise to a conformal
foliation of its domain, and when the target manifold is
a surface, the leaves of this foliation are minimal submanifolds.
Hence any submersive harmonic morphism from a 3-manifold to
a surface gives rise to a conformal foliation by geodesics
of the domain. In a careful study of this situation, P.~Baird
and J.~C.~Wood proved that the Ricci curvature must be conformal on
the distribution orthogonal to the
leaves. It follows from this that there can be at most two
distinct conformal foliations by geodesics of a 3-manifold
with non-constant sectional curvature, and a similar result
is then true for the number of possible harmonic morphisms.
For precise statements of these results, we refer the reader
to Section 10.6 of \cite{Bai-Woo-book} or
Theorem \ref{thm:uniqueness} below. By using this, Baird and
Wood showed that the 3-dimensional Thurston geometry Sol
(see Section 3 below) \emph{does not allow any non-constant
harmonic morphisms to a surface, not even locally}.

We will here apply the results of Baird and Wood on conformal
foliations by geodesics to 3-dimensional Lie groups.
We give a complete classification of those $3$-dimensional
Riemannian Lie groups admitting harmonic morphisms to surfaces.
In particular, we show that Sol fits into a continuous
family of Lie groups, none of which admits any non-constant
local harmonic morphisms to surfaces.

For the general theory of harmonic morphisms, we refer to
the exhaustive book \cite{Bai-Woo-book} or the on-line
bibliography of papers \cite{Gud-bib}.

\section{Harmonic morphisms and minimal conformal foliations}

Let $M$ and $N$ be two manifolds of dimensions $m$ and $n$,
respectively. A Riemannian metric $g$ on $M$ gives rise to the
notion of a {\it Laplacian} on $(M,g)$ and real-valued {\it
harmonic functions} $f:(M,g)\to\rn$. This can be generalized to
the concept of {\it harmonic maps} $\phi:(M,g)\to (N,h)$ between
Riemannian manifolds, which are solutions to a semi-linear system
of partial differential equations, see \cite{Bai-Woo-book}.

\begin{definition}
A map $\phi:(M,g)\to (N,h)$ between Riemannian manifolds is
called a {\it harmonic morphism} if, for any harmonic function
$f:U\to\rn$ defined on an open subset $U$ of $N$ with $\phi^{-1}(U)$
non-empty,
$f\circ\phi:\phi^{-1}(U)\to\rn$ is a harmonic function.
\end{definition}

The following characterization of harmonic morphisms between
Riemannian manifolds is due to Fuglede and Ishihara.  For the
definition of horizontal (weak) conformality we refer to
\cite{Bai-Woo-book}.

\begin{theorem}\cite{Fug-1,T-Ish}
A map $\phi:(M,g)\to (N,h)$ between Riemannian manifolds is a
harmonic morphism if and only if it is a horizontally (weakly)
conformal harmonic map.
\end{theorem}

The next result gives the theory of harmonic morphisms a strong
geometric flavour. It also shows that the
case when the codomain is a surface is particularly interesting.

\begin{theorem}\cite{Bai-Eel}\label{theo:B-E}
Let $\phi:(M^m,g)\to (N^n,h)$ be a horizontally (weakly) conformal
map between Riemannian manifolds. If
\begin{enumerate}
\item[i.] $n=2$, then $\phi$ is harmonic if and only if $\phi$ has
minimal fibres at regular points; \item[ii.] $n\ge 3$, then two of
the following
conditions imply the other:
\begin{enumerate}
\item $\phi$ is a harmonic map, \item $\phi$ has minimal fibres at
regular points,
\item $\phi$ is horizontally homothetic.
\end{enumerate}
\end{enumerate}
\end{theorem}

In particular, the conditions characterizing harmonic morphisms
into a surface $N^2$ only depend on the conformal structure of $N^2$.

\section{The $3$-dimensional Lie groups}\label{section-Bianchi}

At the end of the 19th century, L.~Bianchi classified the $3$-dimensional
real Lie algebras.  They fall into nine disjoint types I-IX.
Each contains a single isomorphy class except types VI and VII which
are continuous families of different classes.  For later reference
we list below Bianchi's classification and notation for the corresponding
simply connected Lie groups. We also equip these Lie groups with the
left-invariant metric for which the given basis $\{X,Y,Z\}$ of each Lie
algebra is orthonormal at the identity.

\begin{example}[Type I]
The Abelian Lie algebra $\rn^3$; the corresponding simply connected
Lie group is of course the Abelian group $\rn^3$ which we equip with
the standard flat metric.
\end{example}

\begin{example}[II] The Lie algebra $\n_3$ with a basis $X,Y,Z$ satisfying
$$[X,Y]=Z.$$
The corresponding simply connected Lie group is the nilpotent Heisenberg
group $\text{Nil}^3$.
\end{example}

\begin{example}[III]
The Lie algebra $\h^2\oplus\rn=\text{span}\{X,Y,Z\}$,
where $\h^2$ is the two-dimensional Lie algebra with basis $X,Y$ satisfying
$$[Y,X]=X.$$
The corresponding simply connected Lie group is denoted by $H^2\times\rn$.
Here $H^2$ is the standard hyperbolic plane.
\end{example}

\begin{example}[IV]
The Lie algebra $\g_4$ with a basis $X,Y,Z$ satisfying
$$[Z,X]=X,\quad [Z,Y]=X+Y.$$
The corresponding simply connected Lie group is denoted by $G_4$.
\end{example}

\begin{example}[V]
The Lie algebra $\h^3$ with a basis $X,Y,Z$ satisfying
$$[Z,X]=X,\quad [Z,Y]=Y.$$
The corresponding simply connected Lie group $H^3$ is the standard
hyperbolic $3$-space of constant sectional curvature $-1$
\end{example}

\begin{example}[VI] The Lie algebra $\sol_\alpha^3$, where $\alpha>0$,
is the Lie algebra with basis $X,Y,Z$ satisfying
$$[Z,X]=\alpha X,\quad [Z,Y]=-Y.$$
The corresponding simply connected Lie group is denoted by
$\text{Sol}_\alpha^3$.  The group Sol mentioned in the introduction
is actually $\text{Sol}_1^3$.
\end{example}

\begin{example}[VII] The Lie algebra $\g_7(\alpha)$, where
$\alpha\in\rn$, is the the Lie algebra with basis $X,Y,Z$ satisfying
$$[Z,X]=\alpha X-Y,\quad [Z,Y]=X+\alpha Y.$$
The corresponding simply connected Lie group
is denoted by $G_7(\alpha)$.
\end{example}

\begin{example}[VIII] The Lie algebra $\slr 2$ with a basis $X,Y,Z$
satisfying
$$[X,Y]=-2Z,\quad [Z,X]=2Y,\quad [Y,Z]=2X.$$
The corresponding simply connected Lie group is denoted by
$\widetilde{\SLR 2}$ as it is the universal cover of the special
linear group $\SLR 2$.
\end{example}

\begin{example}[IX] The Lie algebra $\su 2$ with a basis $X,Y,Z$
satisfying
$$[X,Y]=2Z,\quad [Z,X]=2Y,\quad [Y,Z]=-2X.$$
The corresponding simply connected Lie group is of course $\SU 2$.
This is isometric to the standard $3$-sphere of constant curvature $+1$.
\end{example}

\section{The classification}

In this section the following theorem of Baird and Wood is applied
to get a complete classification of those $3$-dimensional Riemannian
Lie groups which admitting harmonic morphisms to surfaces.

\begin{theorem}\cite{Bai-Woo-book}\label{thm:uniqueness}
Let $M$ be a $3$-dimensional Riemannian manifold with
non-constant sectional curvature. Then there are at most two
distinct conformal foliations by geodesics of $M$. If there is
an open subset on which the Ricci tensor has precisely two distinct
eigenvalues, then there is at most one conformal foliation by
geodesics of $M$.
\end{theorem}

The simply connected $3$-dimensional Riemannian Lie groups of constant
sectional curvature are the standard $\rn^3$, $H^3$ and $\SU 2$, modulo
a constant multiple of the metric.  For the other cases we have the
following result.

\begin{theorem}\cite{Gud-Sve-4}\label{theo:3-foliation}
Let $G$ be a connected $3$-dimensional Lie group with a
left-invariant metric of non-constant sectional curvature. Then any
local conformal foliation by geodesics of a connected open
subset of $G$ can be extended to a global conformal foliation by
geodesics of $G$.  This is given by the left-translation of a
1-parameter subgroup of $G$.
\end{theorem}

\begin{proof}Assume that $\V$ is a conformal foliation by geodesics
of some connected neighbourhood $U$ of the identity element $e$ of
$G$ and denote by $\H$ the orthogonal complement of $\V$. Let
$U'\subset U$ be a connected neighbourhood of $e$ such that $gh\in
U$ for all $g,h\in U'$, and let $U''\subset U'$ be a connected
neigbourhood of $e$ for which $g^{-1}\in U'$ for all $g\in U''$.

For any $g\in G$, we denote by $L_g:G\to G$ left translation by
$G$. Take $g\in U''$ and consider the distribution
$dL_g\V\big\vert_{U'}$, obtained by restricting $\V$ to $U'$ and
translating with $g$. As $L_g$ is an isometry, this is also a
conformal foliation by geodesics of $L_g U'$, which is a connected
neigbourhood of $e$. It is clear from Theorem \ref{thm:uniqueness}
and by continuity, that this distribution must coincide with $\V$
restricted to $L_g U'$. It follows that $d(L_g)_h(\V_h)=\V_{gh}$
for all $g,h\in U''$. In particular we have
$$d(L_g)_e(\V_e)=\V_g\qquad(g\in U'').$$

Define a 1-dimensional distribution $\tilde\V$ on $G$ by
$$\tilde\V_g=(dL_g)_e(\V_e)\qquad(g\in G).$$ Its horizontal
distribution $\tilde\H$ is clearly given by left translation of
$\H_e$. From the above we see that
$$\tilde\V\big\vert_{U''}=\V\big\vert_{U''}.$$ It follows that
$$B^{\tilde\V}\big\vert_{U''}=B^{\V}\big\vert_{U''}=0,$$ and since
$\tilde\V$ is left-invariant, it follows that $B^{\tilde\V}=0$
everywhere, i.e., $\tilde\V$ is totally geodesic. In the same way
we see that $\tilde\V$ is a conformal distribution and, by Theorem
\ref{thm:uniqueness}, we see that $$\tilde\V\big\vert_{U}=\V.$$

This shows that $\V$ extends to a global conformal, totally
geodesic distribution $\tilde\V$, which is left-invariant. By
picking any unit vector $V\in \V_e$, we see that the corresponding
foliation is given by left translation of the 1-parameter subgroup
generated by $V$.
\end{proof}

Let $G$ be a $3$-dimensional Lie group with Lie algebra $\g$ equipped
with a left-invariant Riemannian metric such that $\{X,Y,Z\}$ is an
orthonormal basis for $\g$.  Assume that the $1$-dimensional
left-invariant foliation generated by $Z\in\g$ is minimal and
horizontally conformal i.e. producing harmonic morphisms.
Then it is easily
seen that the bracket relations for $\g$ are of the form
\begin{eqnarray*}
\lb XY&=&xX+yY+zZ,\\
\lb ZX&=&aX+bY,\\
\lb ZY&=&-bX+aY,
\end{eqnarray*}
where $a,b,x,y,z\in\rn$.  In this case the Jacobi identities
for the Lie algebra
$\g$ imply that $$az=0, \ \ ax+by=0,\ \ bx-ay=0.$$
The following three families of $3$-dimensional Lie algebras give a complete
classification.

\begin{example}
With $a=b=0$ we yield a $3$-dimensional family of Lie groups with bracket
relation
\begin{eqnarray*}
\lb XY&=&xX+yY+zZ.
\end{eqnarray*}
If $x=y=z=0$ then the type is I. If $x$ or $y$ non-zero then
we have type III. If $z\neq 0$ and $x=y=0$, then the type is II.
\end{example}

\begin{example}
In the case of $x=y=z=0$ we get semi-direct products $\rn^2\rtimes \rn$
with bracket relations
\begin{eqnarray*}
\lb ZX&=&aX+bY,\\
\lb ZY&=&-bX+aY.
\end{eqnarray*}
If $b\neq 0$ then the Lie algebra is of type VII.
If $b=0$ then the Lie algebra is of type V or of type I if also $a=0$.
\end{example}

\begin{example}
When $x=y=a=0$ we obtain a $2$-dimensional family with the bracket relations
\begin{eqnarray*}
\lb XY&=&zZ,\\
\lb ZX&=&bY,\\
\lb ZY&=&-bX.
\end{eqnarray*}
When $bz<0$ the Lie algebra is of type VIII and of type IX if $bz>0$.
The case when $z=0$ and $b\neq 0$ is of type VII ($\alpha=0$),
and the case when $b=0$ and $z\neq 0$ is of type II.
The case when $b=z=0$ is of type I.
\end{example}

With the above analysis we have proved the following classification result.

\begin{theorem}
Let $G$ be a $3$-dimensional Lie group with Lie algebra $\g$.
Then there exists a left-invariant Riemannian metric $g$ on $G$ and a
left-invariant horizontally conformal foliation on $(G,g)$ by geodesics
if and only if the Lie algebra $\g$ is neither of type IV nor of type VI.
\end{theorem}

Note that in the cases of type I, II, III, V, VII, VIII and IX the
possible left-invariant Riemannian metrics are completely determined
via isomorphisms to the standard examples presented in
Section \ref{section-Bianchi}.


\begin{thebibliography}{99}

\bibitem{Bai-Eel}
P.~Baird and J.~Eells,
{\it A conservation law for harmonic maps},
Geometry Symposium Utrecht 1980,
Lecture Notes in Mathematics {\bf 894}, 1-25, Springer (1981).

\bibitem{Bai-Woo-1}
P.~Baird and J.~C. Wood,
{\it Harmonic morphisms, Seifert fibre spaces and conformal foliations},
Proc. London Math. Soc. {\bf 64} (1992), 170-197.

\bibitem{Bai-Woo-book}
P.~Baird and J.~C. Wood,
{\it Harmonic morphisms between Riemannian manifolds},
London Math. Soc. Monogr. No. {\bf 29},
Oxford Univ. Press (2003).



\bibitem{Fug-1}
B.~Fuglede,
{\it Harmonic morphisms between Riemannian manifolds},
Ann. Inst. Fourier {\bf 28} (1978), 107-144.

\bibitem{Gud-bib}
S.~Gudmundsson,
{\it The Bibliography of Harmonic Morphisms},
{\tt http://www.matematik.lu.se/\\
matematiklu/personal/sigma/harmonic/bibliography.html}


\bibitem{Gud-Sve-4}
S.~Gudmundsson and M.~Svensson
{\it Harmonic morphisms from solvable Lie groups},
Math. Proc. Cambridge Philos. Soc. {\bf 147} (2009), 389--408.

\bibitem{T-Ish}
T.~Ishihara,
{\it A mapping of Riemannian manifolds which preserves harmonic functions},
J. Math. Soc. Japan {\bf 7} (1979), 345-370.

\bibitem{Jacobi}
C.~G.J.~Jacobi, \emph{\"Uber eine L\"osung der partiellen Differentialgleichung $\frac{\partial^2V}{\partial x^2}+\frac{\partial^2V}{\partial y^2}+\frac{\partial^2V}{\partial z^2}=0$}, J. Reine Angew. Math. {\bf 36} (1848), 113--134.

\bibitem{Ville}
M.~Ville, \emph{Harmonic morphisms from Einstein 4-manifolds to Riemann surfaces}, Internat. J. Math. {\bf 14} (2003), 327--337.


\end{thebibliography}
\end{document}